\documentclass[12pt]{amsart}

\usepackage{tgpagella}
\usepackage{euler}
\usepackage[T1]{fontenc}
\usepackage{amsmath, amssymb}
\usepackage[hidelinks]{hyperref}
\usepackage[english]{babel}
\usepackage{mathrsfs}
\usepackage{eucal}
\usepackage[all]{xy}
\usepackage{tikz}
\usepackage{stmaryrd}
\usepackage{amsfonts}
\usepackage{amsmath}
\usepackage{amsthm}
\usepackage{amssymb}
\usepackage{bbm}
\usepackage{mathtools}
\usepackage{babel}
\usepackage{caption}
\usepackage{tikz}
\usepackage{tkz-graph}
\usepackage{geometry}

\usepackage{todonotes}

\newtheorem{thm}{Theorem}[section]
\newtheorem*{thm*}{Theorem}

\newtheorem{fact}[thm]{Fact}

\newtheorem{prop}[thm]{Proposition}
\newtheorem*{prop*}{Proposition}

\newtheorem{cor}[thm]{Corollary}
\newtheorem*{cor*}{Corollary}

\theoremstyle{definition}

\newtheorem*{defn*}{Definition}

\newtheorem{remark}[thm]{Remark}

\newtheorem*{question*}{Question}
\newtheorem*{Pquestion*}{Popa's question}

\newtheorem*{conv*}{Convention}

\def\bb{\mathbb}

\def\bb{\mathbb}

\def\cal{\mathcal}

\makeatletter

\def\dotminussym#1#2{%
  \setbox0=\hbox{$\m@th#1-$}%
  \kern.5\wd0%
  \hbox to 0pt{\hss\hbox{$\m@th#1-$}\hss}%
  \raise.6\ht0\hbox to 0pt{\hss$\m@th#1.$\hss}%
  \kern.5\wd0}

    \title{The Schr\"oder-Bernstein property for normal operators on Hilbert spaces}
\author{Nicolás Cuervo Ovalle, Isaac Goldbring and Netanel Levi}

\address{Department of Mathematics\\ Los Andes University, Bogotá, Colombia,
Cra. 1 \#18a-12}
\email{n.cuervo10@uniandes.edu.co}
\thanks{The first named author was partially supported by NSF grant DMS-2054477. He would also like to thank the UC Irvine Department of Mathematics for their hospitality.}

\address{Department of Mathematics\\University of California, Irvine, 340 Rowland Hall (Bldg.\# 400),
Irvine, CA 92697-3875}
\email{isaac@math.uci.edu}
\urladdr{http://www.math.uci.edu/~isaac}
\thanks{The second-named author was partially supported by NSF grant DMS-2054477.}
\address{Department of Mathematics\\University of California, Irvine, 340 Rowland Hall (Bldg.\# 400),
Irvine, CA 92697-3875}
\email{netanell@uci.edu}
\urladdr{https://sites.google.com/uci.edu/netanelevi/}
\thanks{The third named author was supported by NSF DMS-2052899, DMS-2155211, and
Simons 896624.}

\begin{document}

 \begin{abstract}
We establish that the complete theory of a Hilbert space equipped with a normal operator has the Schr\"oder-Bernstein property.  This answers a question of Argoty, Berenstein, and the first-named author.  We also prove an analogous statement for unbounded self-adjoint operators.
 \end{abstract}
\maketitle
\section{Introduction}

The classical Schr\"oder-Bernstein theorem in set theory states that when $X$ and $Y$ are sets for which there are injections $X\hookrightarrow Y$ and $Y\hookrightarrow X$, then in fact there is a bijection between $X$ and $Y$.  It is natural to ask if this property holds when $X$ and $Y$ are not mere sets but have additional structure on them. One appropriate setting for such a generalization is to consider structures (in the model-theoretic sense) $\mathcal{M}$ and $\mathcal{N}$ in the same language for which there exist \emph{elementary} embeddings $\mathcal{M}\hookrightarrow \mathcal{N}$ and $\mathcal{N}\hookrightarrow \mathcal{M}$ and to ask if in this case $\mathcal{M}$ and $\mathcal{N}$ must be isomorphic.  When this phenomena holds for all models of a first-order theory, we say that theory has the \textbf{Schr\"oder-Bernstein property} or \textbf{SB-property} for short.

A key motivation for studying the SB-property is that, if a theory $T$ satisfies this property, then its models admit a classification in terms of a well-behaved collection of invariants, thereby indicating a robust structural understanding of the models of the theory. This property has been considered in classical model theory \cite{SBpropertyJohn,SBpropertyW-stable} and in the setting of continuous logic in \cite{SBproperty}. It was also studied in various other settings, including those of Banach spaces \cite{ferenczi-galego,galego-quintuples,galego-solutions,gowers,  koszmider}, modules over rings \cite{dehghani}, category theory \cite{laackman}, and in the context of operator algebras \cite{dixmier}.

In \cite{SBproperty}, the authors introduce a weakening of the SB property for metric structures known as the \textbf{SB-property up to perturbations}; the functional analyst will recognize this weakening as the difference, for example, between unitary equivalence and approximate unitary equivalence of operators. The authors of \cite{SBproperty}  show that the complete theory of atomless probability algebras equipped with a generic automorphism is a theory with the SB-property up to perturbations but without the full-fledged SB-property (\cite[Theorem 3.17 and Corollary 5.10]{SBproperty}), whence the former property is a genuine weakening of the latter property in general.

The authors of \cite{SBproperty} posed the question as to whether or not the complete theory of a structure of the form $(\mathcal{H},T)$ has the SB-property, where $\mathcal{H}$ is a Hilbert space and $T$ is a bounded, self-adjoint operator on $\mathcal{H}$ (\cite[Question 2.29]{SBproperty}).  The authors of \cite{SBproperty} managed to show that the complete theory of any such pair $(\mathcal{H},T)$ has the SB-property up to perturbations (\cite[Theorem 2.28]{SBproperty}) but were only able to show that it has the actual SB-property when $T$ has countable spectrum (\cite[Proposition 2.30]{SBproperty}).

In this paper, we give a positive answer to the above question; in fact, one does not even need to assume that the operators are models of the same theory nor that the embeddings between these structures are elementary.  In addition, one does not even need to require the operator to be self-adjoint, but instead can assume that the operator is merely normal.\footnote{A posteriori, if $(\cal H_1,T_1)$ and $(\cal H_2,T_2)$ satisfy the assumptions of the following theorem, then they are necessarily models of the same theory and, moreover, such embeddings are automatically elementary; these statements are a consequence of the fact that spectrally equivalent normal operators are approximately unitarily equivalent.}  We rephrase this positive resolution of the question in more functional analytic terms:

\begin{thm*}\label{thm_unit_equiv}
	Let $\mathcal{H}_1,\mathcal{H}_2$ be Hilbert spaces and let $T_i:\mathcal{H}_i\to\mathcal{H}_i$ be  normal operators for $i=1,2$. Suppose that there exist linear isometries $U_1:\mathcal{H}_1\to\mathcal{H}_2$ and $U_2:\mathcal{H}_2\to\mathcal{H}_1$ such that $T_1U_2=U_2T_2$ and $T_2U_1=U_1T_1$. Then $T_1$ and $T_2$ are unitarily equivalent, that is, there is a unitary transformation $U:\mathcal{H}_1\to \mathcal{H}_2$ such that $UT_1=T_2U$.
\end{thm*}

By applying the Cayley transform, one can use the previous theorem to prove a version for unbounded self-adjoint operators; we do this in the last section.

The naive approach of following the lines of the classical Schröder-Bernstein proof fails due to the extra structure of the Hilbert space, which introduces geometric and operator-theoretic obstacles. Instead, using the direct integral formulation of the spectral theorem,   show that, in a certain sense, if a normal operator $S_2$ lies between two unitarily equivalent normal operators $S_1$ and $S_3$ (in the sense of being a suboperator), then $S_2$ is also unitarily equivalent to $S_3$; see Proposition~\ref{noinvariant}. This is first proved for separable Hilbert spaces via spectral-theoretic arguments and then extended to general Hilbert spaces using model-theoretic methods. The final step follows from the Fuglede-Putnam theorem, combined with additional operator-theoretic analysis.

We would like to thank Alex Berenstein for discovering an error in the first version of this paper and for very useful discussions regarding the current version. The third named author also discussed this problem with Svetlana Jitomirskaya, who pointed out a reduction to the self-adjoint case via polarization, thereby obtaining the more general statement through an alternative argument.
\begin{remark}
Shortly after completing this work, we became aware of a proof of a statement more general than our main result, relying on the Schröder--Bernstein property for C$^*$-algebra representations: if two C$^*$-algebra representations $\rho$ and $\sigma$ are each unitarily equivalent to a subrepresentation of the other, then they are unitarily equivalent. To the best of our knowledge, this statement does not appear in the literature, but it was proved in a MathOverflow thread \cite{MO:SchroederBernsteinReps}. Using it, one can find an alternative and simpler proof of our main theorem. We have decided to keep this preprint on the arXiv, as it contains several other results which may be of independent interest.
\end{remark}
\section{The main theorem}
Throughout this paper, given two Borel measures $\mu_1,\mu_2$ on $\bb R$, we say $\mu_1$ and $\mu_2$ are mutually absolutely continuous, denoted $\mu_1\sim\mu_2$, if $\mu_1\ll\mu_2$ and $\mu_2\ll\mu_1$. In addition, given two bounded operators $T_1,T_2$ acting on Hilbert spaces $\mathcal{H}_1,\mathcal{H}_2$, we will write $\left(\mathcal{H}_1,T_1\right)\simeq\left(\mathcal{H}_2,T_2\right)$ if $T_1$ and $T_2$ are unitarily equivalent, that is, if there exists a unitary transformation $U:\mathcal{H}_1\to\mathcal{H}_2$ such that $T_1=U^{-1}T_2U$.

We remind the reader of the direct integral version of the spectral theorem for normal operators \cite{dixmier}:

\begin{fact}\label{direct_integral_fact}
    Suppose that $T$ is a normal operator on a separable Hilbert space $\mathcal{H}$.  Then there is a Borel probability measure $\mu$ on the spectrum $\sigma(T)$ of $T$ and a measurable family $(\mathcal{H}_\lambda)_{\lambda\in \mathbb{R}}$ of Hilbert spaces such that $T$ is unitarily equivalent to the multiplication operator on the direct integral $\int^\oplus \mathcal{H}_\lambda d\mu(\lambda)$.  Moreover, this direct integral representation of $T$ is a unitary invariant of $T$ in the following sense: Suppose that for $i=1,2$, $T_i$ is a bounded normal operator acting on a separable Hilbert space $\mathcal{H}_i$. Suppose, in addition that we have Borel probability measures $\mu_1,\mu_2$, a $\mu_1$-measurable family of Hilbert spaces $\left(\mathcal{H}_\lambda^1\right)_{\lambda\in\mathbb{R}}$ and a $\mu_2$-measurable family of Hilbert spaces $\left(\mathcal{H}_\lambda^2\right)_{\lambda\in\mathbb{R}}$ such that for $i=1,2$, $T_i$ is unitarily equivalent to the multiplication operator on the direct integral $\int^\oplus\mathcal{H}_\lambda^i\,d\mu_i\left(\lambda\right)$. Then
$\left(\mathcal{H}_1,T_1\right)\simeq\left(\mathcal{H}_2,T_2\right)$ if and only if: $\mu_1\sim\mu_2$ and for $\mu_1$-almost every $\lambda\in\mathbb{R}$, $\dim\mathcal{H}_\lambda^1=\dim\mathcal{H}_\lambda^2$.
\end{fact}

\begin{cor}\label{helpful_cor}
    Let $T:\mathcal{H}\to\mathcal{H}$ be a bounded normal operator acting on a separable Hilbert space $\mathcal{H}$. Suppose that $W\subseteq\mathcal{H}$ is a closed subspace which is invariant under $T$ and under $T^*$. Consider the direct integral representations of $T$ and of $T|_W$ given by
    \begin{center}
        $\int^\oplus\mathcal{H}_\lambda\,d\mu\left(\lambda\right)$,\\$\text{ }$\\
        $\int^\oplus W_\lambda\,d\rho\left(\lambda\right)$
    \end{center}
    respectively. Then $\rho\ll\mu$, and for $\rho$-almost every $\lambda\in\mathbb{R}$, $\dim W_\lambda\leq\dim\mathcal{H}_\lambda$.
\end{cor}
\begin{proof}
    Let us also consider the direct integral representation of $T|_{W^\perp}$:
    \begin{center}
        $\int^\oplus E_\lambda\, d\sigma\left(\lambda\right)$
    \end{center}
    By the invariance of $W$ under both $T$ and $T^*$, we have that $T=T|_W\oplus T|_{W^\perp}$, which implies that the following is another direct integral representation of $T$:
    \begin{center}
        $\int^\oplus\left(W_\lambda\oplus E_\lambda\right)\,d\left(\rho+\sigma\right)\left(\lambda\right)$.
    \end{center}
    By Fact \ref{direct_integral_fact}, this implies that $\rho+\sigma$ and $\mu$ are mutually absolutely continuous, and that for $\mu$-almost every $\lambda\in\mathbb{R}$, the dimensions of $\mathcal{H}_\lambda$ and of $W_\lambda\oplus E_\lambda$ are equal. In particular we obtain that $\rho\ll\mu$ and that for $\rho$-almost every $\lambda\in\mathbb{R}$, $\dim W_\lambda\leq\dim\mathcal{H}_\lambda$, as required.
\end{proof}

\begin{prop}\label{noinvariant}
    Let $T:\mathcal{H}\to\mathcal{H}$ be a bounded normal operator acting on a separable Hilbert space $\mathcal{H}$. Suppose that $W\subseteq V\subseteq\mathcal{H}$ are closed subspaces which are invariant under both $T$ and under $T^*$. Suppose in addition that $\left(\mathcal{H},T\right)\simeq\left(W,T|_W\right)$. Then $\left(\mathcal{H},T\right)\simeq\left(V,T|_V\right)$.
\end{prop}
\begin{proof}
    Let us consider the direct integral representations of $T,T|_V,$ and $T|_W$:
    \begin{center}
        $\int^\oplus\mathcal{H}_\lambda\,d\mu\left(\lambda\right)$\\$\text{ }$\\
        $\int^\oplus V_\lambda\,d\rho\left(\lambda\right)$,\\$\text{}$\\
        $\int^\oplus W_\lambda\,d\sigma\left(\lambda\right)$.
    \end{center}
    By Corollary \ref{helpful_cor}, we have that $\sigma\ll\rho\ll\mu$, $\dim W_\lambda \leq \dim V_\lambda $ for $\sigma$-almost every $\lambda\in \bb R$, and $\dim V_\lambda \leq \dim \cal H_\lambda $ for $\rho$-almost every $\lambda\in \bb R$. On the other hand, by the unitary equivalence of $T$ and of $T|_W$, we have that $\mu$ and $\sigma$ are mutually absolutely continuous and $\dim \cal H_\lambda =\dim W_\lambda $ for $\sigma$-almost every $\lambda\in \bb R$. This implies that $\mu\ll\rho$, whence $\mu$ and $\rho$ are mutually absolutely continuous, and that for $\rho$-almost every $\lambda\in\mathbb{R}$, $\dim \cal H_\lambda =\dim W_\lambda \leq \dim V_\lambda \leq \dim \cal H_\lambda $, whence $\dim \cal H_\lambda=\dim V_\lambda$ for $\rho$-almost every $\lambda\in \bb R$. By Fact \ref{direct_integral_fact}, this implies that $\left(T,\mathcal{H}\right)\simeq\left(T|_V,V\right)$, as required.
\end{proof}

Our next goal is to show that Proposition \ref{noinvariant} holds for arbitrary (that is, not necessarily separable) Hilbert spaces.  To achieve that goal, we employ some model-theoretic techniques.  

In what follows, let $L$ be the language extending the language of Hilbert spaces which contains four new unary function symbols $T,P,Q,$ and $U$, all of which have $1$-Lipshitz moduli of uniform continuity.

\begin{prop}
There is a set $\Sigma$ of $L$-sentences whose models are exactly those $L$-structures of the form $(\cal H,T,P_V,P_W,U)$, where:
\begin{enumerate}
    \item $T:\cal H\to \cal H$ is a normal operator with $\|T\|\leq 1$;
    \item $P_V$ and $P_W$ are orthogonal projections onto closed subspaces $V$ and $W$ respectively, both of which are invariant under $T$ and $T^*$;
    \item $W\subseteq V$;
    \item $U:\cal H\to \cal W$ is a unitary map witnessing that $(\cal H,T)\simeq (W,T|W)$.
\end{enumerate}
\end{prop}

\begin{proof}
Besides axioms stating that $T$, $P$, $Q$, and $U$ are linear and that $T$ is normal, we add the following axioms:
    \begin{enumerate}
    \item$\sup_x \max(d(Px,P^2x),d(Qx,Q^2x))=0$
    \item $\sup_{x.y}|\max(\langle Px,y\rangle-\langle x,Py\rangle|,\langle Qx,y\rangle-\langle x,Qy\rangle|)=0$
    \item $\sup_x d(PQ(x),P(x))=0$
    \item $\sup_x \max(d(PTPx,TPx),d(QTQx,TQx))=0$
    \item $\sup_{x,y}\max(|\langle Px,TPy\rangle-\langle Px,Ty\rangle|,|\langle Qx,TQy\rangle-\langle Qx,Ty\rangle|)=0$
    \item $\sup_{x,y}|\langle Ux,Uy\rangle-\langle x,y\rangle|=0$
    \item $\sup_x d(PUx,Ux)=0$
    \item $\sup_x\inf_y d(Uy,Qx)=0$
    \item $\sup_x d(UTx,TUx)=0$
    \end{enumerate}
    Axioms (1) and (2) state that $P$ and $Q$ are projections, while axiom (3) states that the image of $Q$ is contained in the image of $P$.  Axiom (4) states that the images of $P$ and $Q$ are $T$-invariant while axiom (5) states that they are $T^*$-invariant.  Axiom (6) states that $U$ is an isometry and axiom (7) states that the image of $U$ is contained in the image of $P$.  A priori, axiom (8) merely states that the image of $U$ is dense in the image of $Q$, but since $U$ is an isometry, it follows that it is in fact onto the image of $Q$.  The final axiom expresses that $U$ interwtwines $T$ and the restriction of $T$ to the image of $Q$.  
\end{proof}

\begin{prop}\label{orthogonalcomplements}
Suppose that $\cal M,\cal N\models \Sigma$ and $\cal M\preceq \cal N$.  Write $\cal M=(\cal H_\cal M,T_\cal M,V_\cal M,W_\cal M,U_\cal M)$ and $\cal N=(\cal H_\cal N,T_\cal N,V_\cal N,W_\cal N,U_\cal N)$.  Then:
\begin{enumerate}
    \item $W_\cal N\ominus W_\cal M\subseteq V_\cal N\ominus V_\cal M\subseteq \cal H_\cal N\ominus \cal H_\cal M$.
    \item $U_\cal N|H_\cal M=U_\cal M$.
    \item $U_\cal N|(\cal H_\cal N\ominus \cal H_\cal M)$ witnesses that $$(\cal H_\cal N\ominus \cal H_\cal M,T_\cal N|\cal H_\cal N\ominus \cal H_\cal M)\simeq (W_\cal N\ominus W_\cal M,T_\cal N|W_\cal N\ominus W_\cal M).$$
\end{enumerate}
\end{prop}

\begin{proof}
To prove (1), take $x\in W_\cal N\ominus W_\cal M$ and $y\in V_\cal M$; we need $x\perp y$.  If $y\in W_\cal M$, then this is clear; if $y\in V_\cal M\ominus W_\cal M$, then $(\sup_z |\langle Qz,y\rangle|)^\cal M=0$, whence, by elementarity, $(\sup_z |\langle Qz,y\rangle|)^\cal N=0$, and thus $x\perp y$ in this case as well.  Since $V_\cal M=W_\cal M\oplus (V_\cal M\ominus W_\cal M)$, we have that $x\perp y$ for arbitrary $y\in V_\cal M$.  The argument that $V_\cal N\ominus V_\cal M\subseteq \cal H_\cal N\ominus \cal H_\cal M$ is similar. 

(2) is clear and (3) follows from (1) and (2).
\end{proof}

\begin{prop}\label{noinvariantgeneral}
Proposition \ref{noinvariant} holds for arbitrary Hilbert spaces $\cal H$.
\end{prop}

\begin{proof}
We prove the proposition by induction on the density character\footnote{The density character of a Hilbert space is the smallest cardinality of a dense subset.  For uncountable density characters, this coincides with the Hilbert space dimension; for density character $\aleph_0$, the dimension could be any natural number or $\aleph_0$.} $\kappa$ of $\cal H$.  The case that $\kappa=\aleph_0$ was established in Proposition \ref{noinvariant}.

Now assume that $\kappa$ is an uncountable cardinal and that the proposition holds for Hilbert spaces of smaller density character.  Suppose that $\cal H$ has density character $\kappa$ and that $\cal M=(\cal H,T,V,W,U)$ is a model of $\Sigma$.\footnote{We are implicitly assuming, without loss of generality, that $\|T\|\leq 1$.}  We wish to show that $(\cal H,T)\simeq (V,T|V)$.

Let $(\cal M_\alpha)_{\alpha<\kappa}$ be a sequence of elementary substructures of $\cal M$ such that, for all $\alpha<\kappa$, we have:
\begin{itemize}
    \item $\cal M_\alpha\preceq \cal M_{\alpha+1}$;
    \item $\cal M_\alpha=\bigcup_{\beta<\alpha} \cal M_\beta$ if $\alpha$ is a limit ordinal;
    \item the density character of $\cal M_\alpha$ is less than $\kappa$;
    \item $\bigcup_{\alpha<\kappa}\cal M_\alpha=\cal M$.
\end{itemize}

Write $\cal M_\alpha=(\cal H_\alpha,T_\alpha,V_\alpha,W_\alpha,U_\alpha)$. 

We now construct, by recursion on $\alpha$, maps $\bar U_\alpha:\cal H_\alpha\to \cal H_\alpha$ which witness that $(\cal H_\alpha,T_\alpha)\simeq (V_\alpha,T_\alpha|V_\alpha)$.  Furthermore, we will construct the maps $\bar U_\alpha$ so that $\bar U_\beta|\cal H_\alpha=\bar U_\alpha$ for all $\alpha<\beta$, whence, setting $\bar U$ to be the unique extension of $\bigcup_{\alpha<\kappa}\bar U_\alpha$ to an isometry on $\cal H$, $\bar U$ will witness that $(\cal H,T)\simeq (V,T|V)$, as desired.

Since the density character of $\cal H_0$ is less than $\kappa$, we may assume, by induction, that $\bar U_0$ exists.  When $\alpha$ is a limit ordinal, we let $\bar U_\alpha$  be the unique extension of $\bigcup_{\beta<\alpha} \bar U_\beta$ to an isometry on $\cal H_\alpha$.  It remains to define $\bar U_{\alpha+1}$.  By Proposition \ref{orthogonalcomplements}, we know that $(\cal H_{\alpha+1}\ominus \cal H_\alpha,T_{\alpha+1}|\cal H_{\alpha+1}\ominus \cal H_\alpha)\simeq (W_{\alpha+1}\ominus W_\alpha,T_{\alpha+1}|W_{\alpha+1}\ominus W_\alpha)$; since the density character of $\cal H_{\alpha+1}$ is less than $\kappa$, by induction, we know that there is a $\bar U'_\alpha:\cal H_{\alpha+1}\ominus \cal H_\alpha\to \cal H_{\alpha+1}\ominus \cal H_\alpha$ witnessing that $(\cal H_{\alpha+1}\ominus \cal H_\alpha,T_{\alpha+1}|\cal H_{\alpha+1}\ominus \cal H_\alpha)\simeq (V_{\alpha+1}\ominus V_\alpha,T_{\alpha+1}|V_{\alpha+1}\ominus V_\alpha)$.  Defining $\bar U_{\alpha+1}:=\bar U_\alpha\oplus \bar U'_\alpha$ yields the desired  map.
\end{proof}
The last ingredient we need for the proof of our main theorem is the Fuglede-Putnam theorem (see, for example, \ \cite[Chapter IX]{conway}):

\begin{fact}\label{fuglede}
Suppose that $T_i:\cal H_i\to \cal H_i$ are bounded operators for $i=1,2$.  Further suppose that $S:\cal H_1\to \cal H_2$ is a bounded operator that \emph{interwines} $T_1$ and $T_2$, that is, such that $ST_1=T_2S$. Then $S$ also intertwines $T_1^*$ and $T_2^*$, that is, $ST_1^*=T_2^*S$.  
\end{fact}

 We are now ready to prove our main theorem.  We repeat the statement for the convenience of the reader.

    \begin{thm}\label{maintheorem}
        Let $\mathcal{H}_1,\mathcal{H}_2$ be Hilbert spaces and let $T_i:\mathcal{H}_i\to\mathcal{H}_i$ be bounded normal operators. Suppose that there exist linear isometries $U_1:\mathcal{H}_1\to\mathcal{H}_2$ and $U_2:\mathcal{H}_2\to\mathcal{H}_1$ such that $T_1U_2=U_2T_2$ and $T_2U_1=U_1T_1$. Then $T_1$ and $T_2$ are unitarily equivalent.
    \end{thm}
    \begin{proof}
        Set $V:=U_1\left(\mathcal{H}_1\right)$ and $W:=U_1\left(U_2\left(\mathcal{H}_2\right)\right)$. We claim that $V$ and $W$ are closed subspaces which are invariant under $T_2$ and under $T_2^*$. Closedness follows immediately from unitarity of $U_1$ and $U_2$. To see that these spaces are $T_2$-invariant, let us consider $\varphi_1\in V$, $\varphi_2\in W$ and take $\psi_1\in\mathcal{H}_1$, $\psi_2\in\mathcal{H}_2$ such that $\varphi_1=U_1\psi_1$, $\varphi_2=U_1U_2\psi_2$. Then we have
        \begin{center}
            $T_2\varphi_1=T_2U_1\psi_1=U_1T_1\psi_1\in U_1\left(\mathcal{H}_1\right)=V$,\\$\text{}$\\
            $T_2\varphi_2=T_2\left(U_1U_2\psi_2\right)=U_1T_1U_2\psi_2=U_1U_2T_2\psi_2\in U_1U_2\left(\mathcal{H}_1\right)=W$,
        \end{center}
        as required. It is left to show that both spaces are $T_2^*$-invariant. Note that the only thing we used to prove the invariance under $T_2$ is the intertwining properties of $T_1$ and $T_2$ under $U_1$ and $U_2$. By Fact \ref{fuglede}, this property also holds for $T_1^*$ and $T_2^*$. Consequently, one can proceed along the same lines and obtain the invariance of $W$ and $V$ under $T_2^*$. 
        
        We now claim that $\left(\mathcal{H}_2,T_2\right)\simeq\left(W,T_2|_W\right)$. This is true since $U_1U_2$ is a surjective isometry between $\mathcal{H}_2$ and $W$, which means it is unitary. In addition, by the above calculation it commutes with $T_2$. An identical argument implies that $\left(\mathcal{H}_1,T_1\right)\simeq\left(V,T_2|_V\right)$. Finally, by Proposition \ref{noinvariantgeneral}, we get that $\left(\mathcal{H},T_2\right)\simeq\left(V,T_2|V\right)$. By transitivity of the relation of unitary equivalence, we obtain $\left(\mathcal{H}_1,T_1\right)\simeq\left(\mathcal{H}_2,T_2\right)$ as required.
    \end{proof}
\section{Unbounded self-adjoint operators}

\begin{thm}
    Let $\mathcal{H}_1$ and $\mathcal{H}_2$ be Hilbert spaces and let $T_i:\mathcal{H}_i\to \mathcal{H}_i$ be unbounded self-adjoint operators.  Further suppose that there exist linear isometries $U_1:\mathcal{H}_1\to \mathcal{H}_2$ and $U_2:\mathcal{H}_2\to \mathcal{H}_1$ such that $U_1T_1\subseteq T_2U_1$ and $U_2T_2\subseteq T_1U_2$.  Then $(\mathcal{H}_1,T_1)\simeq (\mathcal{H}_2,T_2)$ in the sense that there is a unitary transformation $U:\mathcal{H}_1\to \mathcal{H}_1$ such that $UT_1U^{-1}=T_2$.
\end{thm}
\begin{proof}
    For each $i=1,2$, let $V_i:=(T_i-i)(T_i+i)^{-1}$ be the Cayley transform of $T_i$, which is a unitary operator on $\mathcal{H}_i$.  Note then that we have
    \begin{alignat}{2}
U_1V_1&=U_1(T_1-i)(T_1+i)^{-1}\notag \\ \notag
                     &=(T_2-i)U_1(T_1+i)^{-1}\\ \notag
                     &=(T_2-i)(T_2+i)^{-1}U_1\\ \notag
                     &=V_2U_1.\notag
    \end{alignat}
    
   In the same way, we have that $U_2V_2=V_1U_2$.  Since each $V_i$ is unitary (and hence normal), the proof of Theorem \ref{maintheorem},  shows that $U_1$ is a surjective isometry.  It follows then that $U_1T_1U_1^{-1}\subseteq T_2$ and thus $U_1T_1U_1^{-1}=T_2$ as self-adjoint operators have no proper symmetric extensions.
\end{proof}


\begin{thebibliography}{99}

\bibitem{SBproperty} C. Argoty, A. Berenstein, and N. Cuervo Ovalle, \emph{The SB-property on metric structures}, Archive for Mathematical Logic, 2025, pp. 1–29.

\bibitem{conway} J.B. Conway, \emph{A Course in Functional Analysis}, Graduate Texts in Mathematics, vol. 96, Springer-Verlag, 1990.

\bibitem{dehghani} N. Dehghani, \emph{On the Schr\"oder–Bernstein property for modules}, Journal of Algebra \textbf{525} (2019), 1–15.

\bibitem{dixmier} J. Dixmier, \emph{Les alg\`ebres d'op\'erateurs dans l'espace Hilbertien}, Gauthier-Villars, Paris, 1969.

\bibitem{ferenczi-galego} V. Ferenczi and E.M. Galego, \emph{Some results on the Schr\"oder-Bernstein property for separable Banach spaces}, Canadian Journal of Mathematics \textbf{59} (2007), no. 3, 569–584.

\bibitem{galego-quintuples} E.M. Galego, \emph{Schr\"oder-Bernstein quintuples for Banach spaces}, Bulletin of the Polish Academy of Sciences. Mathematics \textbf{54} (2006), no. 2, 113–124.

\bibitem{galego-solutions} E.M. Galego, \emph{On solutions to the Schr\"oder-Bernstein problem for Banach spaces}, Archiv der Mathematik \textbf{79} (2002), no. 4, 299–307.

\bibitem{gowers} W.T. Gowers, \emph{A solution to the Schr\"oder–Bernstein problem for Banach spaces}, Bulletin of the London Mathematical Society \textbf{28} (1996), no. 3, 297–304.

\bibitem{SBJohnthesis} J. Goodrick, \emph{When are elementarily bi-embeddable models isomorphic?}, Ph.D. thesis, University of California, Berkeley, 2007.

\bibitem{SBpropertyJohn} J. Goodrick and M.C. Laskowski, \emph{The Schr\"oder-Bernstein property for a-saturated models}, Proceedings of the American Mathematical Society \textbf{142} (2014), 1013–1023.

\bibitem{koszmider} P. Koszmider, \emph{A C(K) Banach space which does not have the Schr\"oder–Bernstein property}, Fundamenta Mathematicae \textbf{213} (2011), 197–220.

\bibitem{laackman} B. Laackman, \emph{The Cantor–Schr\"oder–Bernstein Property in Categories}, University of Chicago REU Report, 2010. Available at: \url{https://www.math.uchicago.edu/~may/VIGRE/VIGRE2010/REUPapers/Laackman.pdf}

\bibitem{SBpropertyW-stable} T.A. Nurmagambetov, \emph{Characterization of $\omega$-stable theories of bounded dimension}, Algebra and Logic \textbf{28} (1989), 388–396.

\bibitem{MO:SchroederBernsteinReps}
Andreas Thom, answer to MathOverflow question “Schröder–Bernstein for representations of operator algebras,” August 8, 2025.\url{https://mathoverflow.net/questions/303401/schröder-bernstein-for-representations-of-operator-algebras}.

\end{thebibliography}
\end{document}